\documentclass{article}

\usepackage[dvips]{graphics}
\usepackage{amsfonts}
\usepackage{amssymb}
\usepackage{amsthm}
\usepackage{times}

\usepackage[dvips]{graphics}

\theoremstyle{plain}
\newtheorem{theorem}{Theorem}

\newtheorem{proposition}[theorem]{Proposition}

\theoremstyle{definition}
\newtheorem{definition}[theorem]{Definition}

\newtheorem{remark}[theorem]{Remark}

\newcommand{\todo}[1]{\vspace{5 mm}\par \noindent
\marginpar{\textsc{ToDo}}
\framebox{\begin{minipage}[c]{0.95 \textwidth}
\tt #1 \end{minipage}}\vspace{5 mm}\par}

\renewcommand{\todo}[1]{}

\newcommand{\idiot}[1]{\vspace{5 mm}\par \noindent
\framebox{\begin{minipage}[c]{0.95 \textwidth}
\tt #1 \end{minipage}}\vspace{5 mm}\par}

\renewcommand{\idiot}[1]{}

\newcommand{\xs}{x_1,\ldots,x_n}                

\newcommand{\depth}{{\rm{depth}}\ }             
\newcommand{\dimn}{ {\rm{dim}} \ }              
\newcommand{\F}{{\mathcal{F}}}                  
\newcommand{\N}{{\mathcal{N}}}                  
\newcommand{\D}{\Delta}                         
              
\newcommand{\tuple}[1]{\langle #1 \rangle}      

\newcommand{\st}{\ | \ }
  
\newcommand{\Fc}{\mbox{Facets}}               
\newcommand{\height}{\mbox{height }}               

\title{The projective dimension of sequentially Cohen-Macaulay monomial ideals} \author{Sara Faridi\thanks{Department of
Mathematics and Statistics, Dalhousie University, Halifax, Canada, 
faridi@mathstat.dal.ca, +1(902)-494-2658. 
Research supported by NSERC.}} 

\begin{document}

\maketitle


\emph{Since posting this paper we have found that the same result
  regarding projective dimension of square-free monomial ideals
  appears in the paper of Morey and Villarreal~\cite{MV}.}

\bigskip 

In this short note we prove that the projective dimension of a
sequentially Cohen-Macaulay square-free monomial ideal is equal to the
maximal height of its minimal primes (also known as the big height),
or equivalently, the maximal cardinality of a minimal vertex cover of
its facet complex. Along the way we also give bounds for the depth and
the dimension of any monomial ideal. This in particular gives a
formula for the projective dimension of facet ideals of these classes
of ideals, which are known to be sequentially Cohen-Macaulay: graph
trees and simplicial trees and forests~\cite{F1}, chordal graphs and
some cycles~\cite{FV}, chordal clutters and graphs~\cite{W}, some path
ideals~\cite{SKT} to mention a few.  Since polarization preserves
projective dimension, our result also gives the projective dimension
of any sequentially Cohen-Macaulay monomial ideal.  Our result is a
precise and simple description of the projective dimension that is not
dependent on the ground field.

There has been much activity surrounding combinatorial
characterizations of the projective dimension, see for
example~\cite{C,DHS,DS,Ku,LM}. The authors in~\cite{KhM} and~\cite{DS}
in particular prove the same result for graphs that are forests and
some other classes of graphs.

A {\bf simplicial complex} $\Delta$ over a set of vertices $V$ is a
set of subsets of $V$ with the property that if $F \in \Delta$ then
all subsets of $F$ are also in $\Delta$. An element of $\Delta$ is
called a {\bf face}, the {\bf dimension} of a face $F$ is $|F| -1$,
and the dimension of $\D$ is the largest dimension of a face of $\D$.
The maximal faces of $\Delta$ under inclusion are called {\bf facets},
and the set of facets of $\D$ is denoted by $\Fc(\D)$. If
$\Fc(\D)=\{F_1,\ldots,F_q\}$ we write $\D=\tuple{F_1,\ldots,F_q}$.

 Let $k$ be any field.  To a square-free monomial ideal $I$ in a
 polynomial ring $R=k[\xs]$ one can associate two unique simplicial
 complexes $\N(I)$ and $\F(I)$ on the vertex set labeled
 $\{\xs\}$. Conversely given a simplicial complex $\D$ with vertices
 labeled $\xs$ one can associate two unique square-free monomials
 $\N(\D)$ and $\F(\D)$ in the polynomial ring $k[\xs]$; these are all
 defined below.

 \bigskip
\begin{tabular}{ll}
{\bf Facet complex} of $I$  &$\F(I)=
  \tuple{\{x_{a_1},\ldots,x_{a_m}\} \st x_{a_1}\ldots x_{a_m} \mbox{
      minimal generator of }I}$\\
{\bf Stanley-Reisner complex} of $I$  &$\N(I)= \{\{x_{a_1},\ldots,x_{a_m}\} \st x_{a_1}\ldots x_{a_m} \notin I\}.$\\
 {\bf Facet ideal} of $\D$ &$\F(\D)= (x_{a_1}\ldots x_{a_m} \st \{x_{a_1},\ldots,x_{a_m}\} \in \Fc(\D))$\\
 {\bf Stanley-Reisner ideal} of $\D$ &$\N(\D)=(x_{a_1}\ldots x_{a_m} \st \{x_{a_1},\ldots,x_{a_m}\} \notin \D).$
\end{tabular}
 \bigskip

Sequentially Cohen-Macaulay ideals were  introduced by
Stanley in relation to  the concept of nonpure shellability by
Bj\"{o}rner and Wachs [BW]. 

\begin{definition}[{[S]} Chapter III, Definition~2.9]\label{scm-def} 
  Let $M$ be a finitely generated ${\mathbb{Z}}$-graded module over a
  finitely generated ${\mathbb{N}}$-graded $k$-algebra, with $R_0=k$.
  We say that $M$ is {\bf sequentially Cohen-Macaulay} if there
  exists a finite filtration
  $$0=M_0 \subseteq M_1 \subseteq \ldots \subseteq M_r=M$$
  of $M$ by
  graded submodules $M_i$ satisfying the following two conditions.
\begin{enumerate}
\item[(a)] Each quotient $M_i/M_{i-1}$ is Cohen-Macaulay.
\item[(b)] $\dimn (M_1/M_0) < \dimn (M_2/M_1) < \ldots < \dimn (M_r/M_{r-1})$, 
where $\dimn$ denotes Krull dimension.
\end{enumerate}
\end{definition}

Our goal is to find the projective dimension of sequentially Cohen-Macaulay 
square-free monomial ideal. In this setting Duval gave 
an equivalent characterization in terms of simplicial complexes.

Given a simplicial complex $\D$  and an integer $i$ let
\begin{itemize}
\item $\D^i=\{F\in \D \st \dim F \leq i\}$ =$i$-skeleton of $\D$, 
\item $\D_i=\tuple{F \in \D \st \dim F=i}$ =pure $i$-skeleton of $\D$.  
\end{itemize}

\begin{theorem}[{[D]  Theorem~3.3}]\label{t:scm} Let $I$ be
  square-free monomial ideal $I$ in a polynomial ring $R$ over a field
  $k$, and let $\D=\N(I)$. Then $R/I$ is sequentially Cohen-Macaulay
  if and only if $R/\N(\D_i)$ is Cohen-Macaulay for all $i \in \{-1,\ldots,
  \dimn \D\}$.
\end{theorem}

The main tool we will use is a result of Fr\"oberg. For a given
simplicial complex $\D$, let 

\begin{proposition}[\cite{Fr} Theorem~8]\label{p:fr} Let $\D$ be the Stanley-Reisner complex of
I. Then $$\depth(R/I)=\max\{i \st \D^i \mbox{ is Cohen-Macaulay }\}+1.$$
\end{proposition}

\begin{theorem} Let $I$ be a monomial ideal in the polynomial ring 
$R=k[\xs]$. Suppose
  that the maximal height of an associated prime of $I$ is $d$. Then
  $$\depth(R/I)\leq n-d \mbox{ and }\mbox{pd}(R/I) \geq d.$$  
In particular, if $R/I$ is sequentially Cohen-Macaulay then 
  $$\depth(R/I)=n-d \mbox{ and }\mbox{pd}(R/I)=d.$$  
\end{theorem}

\begin{proof}
If $I$ is a monomial ideal it has the same projective dimension as its
polarization, which is square-free. Moreover by Proposition~2.5
of~\cite{F2} the maximal height of an associated prime of $I$ is the
same as the maximal height of an associated prime of its polarization,
so we only consider square-free monomial ideals.

Suppose $I$ is a square-free monomial ideal in $R=k[\xs]$ with primary
decomposition written as $$I=(p_1\cap\cdots \cap p_{a_1}) \cap \cdots
\cap (p_{a_{s-1}+1} \cap \cdots \cap p_{a_s})$$ where
\begin{itemize}
\item $\height p_{a_{i-1}+1}= \cdots = \height p_{a_i}=d_i$ for $i\in\{1,\ldots,s\}$, assuming $a_0=0$, and 
\item $d_1 < d_2 < \ldots <d_s$.
\end{itemize}

We know then that $\dimn (R/I)=n-d_1$. We wish to prove that if $R/I$
is sequentially Cohen-Macaulay, then $\depth (R/I)=n-d_s$.

Note that since $I$ is a square-free monomial ideal, each of the $p_i$ is
generated by a subset of $\{\xs\}$ and the generating set of each
$p_i$ corresponds uniquely to a minimal vertex cover of $\F(I)$; that
is, a set of vertices $A$ such that every facet of $\F(I)$ contains
one of the elements of this set, and no proper subset of $A$ has this
property.

Suppose $\D$ is the Stanley-Reisner complex of $I$. Then by the
Proposition~2.4 of~\cite{F1} the facets of are complements of the
$p_i$ and therefore $\D$ have dimensions $$n-d_1-1>\ldots >n-d_s-1.$$

Now consider $\D^i$.  If $ i>n-d_s-1$, then $\D^i$ will have facets of
dimension greater than $n-d_s-1$ as well as facets of dimension
$n-d_s-1$. So the largest $i$ where $\D^i$ is pure is $n-d_s-1$. Since
$\D^i$ being pure is a necessary condition for the Cohen-Macaulayness
of $R/\N(\D^i)$, we already know that $R/\N(\D^i)$ is not
Cohen-Macaulay for $i>n-d_s-1$. It follows immediately from
Proposition~\ref{p:fr} that $$\depth(R/I) \leq n-d_s.$$ Now by
applying the Auslander-Buchsbaum formula~\cite{AB} we have
$$\mbox{pd}(R/I)=n-\depth(R/I) \geq d_s.$$

Also notice that $\D^{n-d_s-1}=\D_{n-d_s-1}$, so if $I$ is
sequentially Cohen-Macaulay by Theorem~\ref{t:scm} we know
$R/\N(\D_{n-d_s})$ is Cohen-Macaulay. It follows that
$\depth(R/I)=n-d_s$.

The Auslander-Buchsbaum formula~\cite{AB} gives
$$\mbox{pd}(R/I)=n-\depth(R/I)=d_s.$$
\end{proof}

\begin{remark} It is  natural to expect that the same statement holds 
for any sequentially Cohen-Macaulay module, since their primary
components behave in the same way as described in the proof above; see
the appendix and in particular Theorem~A4 of~\cite{F2} for more
details of primary decomposition of sequentially Cohen-Macaulay
modules.  Indeed, after showing our result to J\"urgen Herzog he
pointed out that this follows from his joint work in~\cite{HP}.
\end{remark}

\idiot{Dear Sara, greetings from Osaka. This is a nice observation and
indeed gives the projective dimension of quite lot of classes of
monomial ideals. If you have a short proof of this, then this would
be certainly nice.
On the other hand your observation is true for any sequentially
Cohen-Macaulay module as follows from Prop.2.5 and Cor, 2.6 of my
paper with Popescu attached to this e-mail.

Sara: See also BP 3.3.10}


\end{document}